\documentclass[paper=a4,english,fontsize=11pt,parskip=half,abstract=true]{scrartcl}
\usepackage{babel}
\usepackage[utf8]{inputenc}
\usepackage[T1]{fontenc}
\usepackage[left=20mm,right=20mm,top=30mm,bottom=30mm]{geometry}
\usepackage{amsmath}
\usepackage{amsthm}
\usepackage{amssymb}
\usepackage{enumerate} 
\usepackage{thmtools}
\usepackage{mathtools}
\mathtoolsset{centercolon} 
\usepackage{booktabs}
\usepackage{longtable}
\usepackage[bookmarks=true,
            pdftitle={On a bound of Cocke and Venkataraman},
            pdfauthor={Benjamin Sambale},
            pdfkeywords={},
            pdfstartview={FitH}]{hyperref}

\newtheorem{Thm}{Theorem} 
\newtheorem{Lem}[Thm]{Lemma}
\newtheorem{Prop}[Thm]{Proposition}
\newtheorem{Cor}[Thm]{Corollary}
\numberwithin{equation}{section}

\allowdisplaybreaks[1]

\renewcommand{\phi}{\varphi}
\newcommand{\C}{\mathrm{C}}
\newcommand{\N}{\mathrm{N}}
\newcommand{\Z}{\mathrm{Z}}

\newcommand{\ZZ}{\mathbb{Z}}

\newcommand{\Aut}{\operatorname{Aut}}

\newcommand{\SL}{\operatorname{SL}}

\title{On a bound of Cocke and Venkataraman}
\author{Benjamin Sambale and Philipp Wellmann\footnote{Institut für Algebra, Zahlentheorie und Diskrete Mathematik, Leibniz Universität Hannover, Welfengarten 1, 30167 Hannover, Germany,
\href{mailto:sambale@math.uni-hannover.de}{sambale@math.uni-hannover.de}, \href{mailto:wellmann@stud.uni-hannover.de}{wellmann@stud.uni-hannover.de}}}
\date{\today}

\begin{document}
\frenchspacing
\maketitle
\begin{abstract}\noindent
Let $G$ be a finite group with exactly $k$ elements of largest possible order $m$. Let $q(m)$ be the product of $\gcd(m,4)$ and the odd prime divisors of $m$. We show that $|G|\le q(m)k^2/\phi(m)$ where $\phi$ denotes Euler's totient function. This strengthens a recent result of Cocke and Venkataraman. As an application we classify all finite groups with $k<36$. This is motivated by a conjecture of Thompson and unifies several partial results in the literature.
\end{abstract}


\section{Introduction}

Let $G$ be a finite group with maximal element order $m:=\max\{|\langle x\rangle|:x\in G\}$. Recently, Cocke and Venkataraman~\cite{CockeVen} have shown that the order of $G$ is bounded by a function on the number $k$ of elements of order $m$ in $G$. More precisely, $\phi(m)$ divides $k$ and $|G|\le mk^2/\phi(m)$ (here and in the following $\phi$ denotes Euler's totient function). The authors have noticed that their bound is best possible for the Frobenius group $G=C_p\rtimes C_{p-1}$ where $p$ is a prime. We have observed conversely that sharpness can only hold if $m$ is neither divisible by $8$ nor by the square of an odd prime. 
In fact, our main theorem gives a much stronger bound on $|G|$. To state it, we introduce the following notation.

Let $p_1,\ldots,p_r$ be the distinct prime divisors of a positive integer $n$. We define $q(n):=p_1\ldots p_r$ if $4\nmid n$ and $q(n):=2p_1\ldots p_r$ if $4\mid n$ (note that $q(1)=1$, the empty product). 

\begin{Thm}\label{main}
Let $G$ be a finite group with exactly $k$ elements of maximal order $m$. Then
\[|G|\le \frac{q(m)k^2}{\phi(m)}.\]
\end{Thm}

We remark that Cocke and Venkataraman's theorem has been extended but not strengthened in \cite{BretCocke}.

In the second part of the paper we enumerate finite groups with $k<36$ by computer calculations. This is motivated by the following conjecture of Thompson: If $G$ and $H$ are finite groups with the same multiset of element orders, then $G$ is solvable if and only if $H$ is solvable. Our classification unifies and completes many partial results in the literature for $k=20,22,24,28,30,\ldots$ (see references).\nocite{maxel20,maxel22,maxel24,maxel28,maxel30,maxel42,maxel44,maxel52,maxelvar} There are exactly $13007$ groups with $k<36$, the largest have order $3264$.

The paper is based on the bachelor's thesis of the second author written under the supervision of the first author.

\section{Proof of the main theorem}

Let $x\in G$ be an element of maximal order. In \cite{CockeVen}, the authors have shown that every coset of $\langle x\rangle$ in $\C_G(x)$ contains at least $\phi(m)$ elements of order $m$. To prove our stronger result, we show more generally that “most” cosets of $\langle x\rangle$ in $\N_G(\langle x\rangle)$ contain at least $\phi(m)$ elements of order $m$. 

\begin{proof}[Proof of \autoref{main}.]
Let $X:=\langle x\rangle\le G$ be a cyclic subgroup of maximal order $m$. Each conjugate of $X$ contains exactly $\phi(m)$ elements of order $m$. Hence, $|G:\N_G(X)|\phi(m)\le k$. 
Let $m=p_1^{a_1}\ldots p_r^{a_r}$ be the prime factorization of $m$ where $p_1<\ldots<p_r$. It is well-known that there is an injective homomorphism
\[\Gamma:\N_G(X)/\C_G(X)\to\Aut(X)\cong(\ZZ/m\ZZ)^\times\cong\bigtimes_{i=1}^r(\ZZ/p_i^{a_i}\ZZ)^\times.\]
For $p_i>2$, the group $(\ZZ/p_i^{a_i}\ZZ)^\times$ is cyclic of order $\phi(p_i^{a_i})=(p-1)p_i^{a_i-1}$ and its Sylow $p$-subgroup is generated by $1+p_i+p_i^{a_i}\ZZ$. On the other hand, for $p=2$ we have $(\ZZ/2^a\ZZ)^\times=\langle -1+2^a\ZZ\rangle\times\langle 5+2^a\ZZ\rangle$. If $p_1>2$, let $N/\C_G(X)$ be the preimage of $\bigtimes_{i=1}^r\langle 1+p_i+p_i^{a_i}\ZZ\rangle$ under $\Gamma$. Otherwise, let $N/\C_G(X)$ be the preimage of
\[\langle 5+2^{a_1}\ZZ\rangle\times\bigtimes_{i=2}^r\langle 1+p_i+p_i^{a_i}\ZZ\rangle.\]
In either case, $N/\C_G(X)$ is cyclic. Moreover, $|\N_G(X):N|\le (p_1-1)\ldots(p_r-1)$ if $4\nmid m$, and $|\N_G(X):N|\le 2(p_1-1)\ldots(p_r-1)$ otherwise.

We will show that every coset of $X$ in $N$ contains at least $\phi(m)$ elements of order $m$. It will then follow that $|N:X|\phi(m)\le k$ and 
\[|G|=|G:\N_G(X)||\N_G(X):N||N:X||X|\le\frac{mk^2}{\phi(m)^2}|\N_G(X):N|\le\frac{q(m)k^2}{\phi(m)}.\]
Let $y\in N$ and $x^y=y^{-1}xy=x^s$ for some $s\in\ZZ$. 
We aim to determine integers $\alpha$ such that $yx^\alpha$ has order $m$. 
Let $1\le i\le r$ and let $q:=m/p_i$. 
The choice of $N$ shows that $q(m)$ divides $s-1$. This implies
\[1+s+\ldots+s^{q-1}=\frac{s^q-1}{s-1}=\sum_{l=1}^q\binom{q}{l}(s-1)^{l-1}\equiv q\pmod{m}.\]
It follows that
\[(yx^\alpha)^q=y^q(x^\alpha)^{y^{q-1}}(x^\alpha)^{y^{q-2}}\ldots x^\alpha=y^qx^{(s^{q-1}+s^{q-2}+\ldots+1)\alpha}=y^qx^{q\alpha}.\]
Note that $y^q\in\C_G(X)$. It is easy to see that there are at least $\phi(p_i^{a_i})$ choices for $\alpha\pmod{p_i^{a_i}}$ such that the $p_i$-factor of $y^qx^{q\alpha}$ is non-trivial, i.\,e. the order of $yx^{\alpha}$ is divisible by $p_i^{a_i}$. 
By the Chinese remainder theorem, these choices can be made independently for each $i$. In total we obtain at least $\phi(m)=\phi(p_1^{a_1})\ldots\phi(p_r^{a_r})$ elements of order $m$ in the coset $yX$. 
\end{proof}

\section{Finite groups with few elements of maximal order}

As before, let $G$ be a finite group with exactly $k$ elements of maximal order $m$. We fix some $X=\langle x\rangle\le G$ of order $m$. To get more precise results we start by analyzing the proof of \autoref{main}. Suppose that we have equality $|G|=q(m)k^2/\phi(m)$. Then we obtain:
\begin{enumerate}[(i)]
\item $|G:\N_G(X)|=k/\phi(m)$, i.\,e. the cyclic subgroups of order $m$ are conjugate in $G$.

\item $|\N_G(X):N|=2^i(p_1-1)\ldots(p_r-1)$ (where $i=1$ if $4\mid m$ and $0$ otherwise). In particular, the numbers $p_i-1$ divide $|G|$.

\item All elements of order $m$ lie in $N$ and every coset of $X$ in $N$ contains exactly $\phi(m)$ elements of order $m$. 
Thus, $N\unlhd G$. 
Let $y\in\C_G(X)$. Then $\langle x,y\rangle$ is abelian and there exists $y'$ such that $\langle x,y\rangle=\langle x\rangle\times\langle y'\rangle$. Since $y'X$ has only $\phi(m)$ elements of order $m$, the order of $y'$ must divide $p_1^{a_1-1}\ldots p_r^{a_r-1}$. In particular, $|\C_G(X):X|$ divides $p_1^{a_1-1}\ldots p_r^{a_r-1}$. Thus, if $m$ is squarefree (i.\,e. $a_1=\ldots =a_r=1$), then $\C_G(X)=X$ and $mk^2/\phi(m)=|G|=|G:\C_G(X)||X|\le km$. Consequently, $k=\phi(m)$ and $G=\N_G(X)$. We extend this observation to the case $q(m)=m$. This characterizes groups with equality in Cocke--Venkataraman's original bound.
\end{enumerate}

\begin{Prop}\label{sharp}
With the notation above, suppose that $|G|=mk^2/\phi(m)$. Then
\begin{enumerate}[(i)]
\item $m=q(m)$,
\item $k=\phi(m)$,
\item $G$ has a unique cyclic subgroup $X$ of order $m$. 
\item $\C_G(X)=X$ and $G/X\cong\Aut(X)$.
\end{enumerate}
\end{Prop}
\begin{proof}
Since $q(m)\le m$, the claim $m=q(m)$ follows from \autoref{main}. The analysis above implies $\C_G(X)=N\unlhd G$ and $|N|=mk/\phi(m)$.  
It follows that $|G:\C_G(X)|=k$, i.\,e. all elements of order $m$ are conjugate in $G$. Hence, for every element $y\in N$ of order $m$ we obtain $\C_G(y)=N$. Therefore, $N$ is abelian. 
The above arguments show further that $N=X\times Y$ where $Y$ is an elementary abelian $2$-group. If $4\nmid m$, then $Y=1$. Now suppose $4\mid m$ and let $X_2$ be the Sylow $2$-subgroup of $X$. 
For every $y\in Y$ there exists $g\in G$ such that $x^g=xy$. This yields $g\in\C_G(x^2)$. Hence, $H:=\C_G(x^2)$ acts on $X_2\times Y$ and permutes the elements of order $4$ transitively. The number of those elements is $2|Y|$ which in turn must divide $|H/N|$. Let $g\in H\setminus N$ be a $2$-element. Then $g$ must be an involution, because otherwise $x^4g$ would be an element of order $\ge m$ outside $N$. For any $z\in X_2\times Y$ also $gz$ is a $2$-element, since $(gz)^2=z^gz\in X_2\times Y$. By the same argument, $gz$ is an involution and therefore $z^g=z^{-1}$. But this completely determines the action of $g$ on $X_2\times Y$. Since $g\in H=\C_G(x^2)$ and $X=\langle x^2\rangle X_2$, also the action of $g$ on $X$ is uniquely determined.
Consequently, there is only one non-trivial $2$-element in $H/N$. In particular, $|H/N|$ is not divisible by $4$. This leads to $Y=1$ and $N=X$. 
Finally, $m=|X|=mk/\phi(m)$ yields $k=\phi(m)$. 
\end{proof}

For $q(m)=m$ it is easy to show that the holomorph $G:=C_m\rtimes\Aut(C_m)$ does indeed have a unique maximal cyclic subgroup of order $m$. Usually, many non-split extensions of $C_m$ by $\Aut(C_m)$ fulfill the conditions of \autoref{sharp} as well. On the other hand, for $m=4$ the quaternion group $Q_8$ has six (and not two) elements of order $4$. 

Since the classification of $p$-groups is particularly complicated we improve \autoref{main} for $p$-groups as follows.

\begin{Prop}\label{podd}
Let $G$ be a $p$-group with exactly $k$ elements of order $p^e=\exp(G)$. Let $k_p$ be the $p$-part of $k$. Then
\[|G|\le\begin{cases}
p^{2-e+\lfloor\log_p(k/(p-1))\rfloor}k_p&\text{ if }p>2,\\
2^{3-e+\lfloor\log_2(k)\rfloor}k_2&\text{ if }p=2.
\end{cases}\]
Moreover, if $e=2$, then $|G|\le 4k$ for $p=2$ and $|G|\le 9k/2$ for $p=3$.
\end{Prop}
\begin{proof}
The $k/\phi(p^e)$ cyclic subgroups of order $p^e$ of $G$ distribute into orbits under the conjugation action. Since each orbit size divides $|G|$, we find a cyclic subgroup $X$ of order $p^e$ such that $|G:\N_G(X)|\le k_p/p^{e-1}$. 
Let $N\le\N_G(X)$ be as in the proof of \autoref{main}. For $p>2$ we have $N=\N_G(X)$ and for $p=2$ we have $|\N_G(X):N|\le 2$.
As in the proof of \autoref{main}, $|N|\le kp^e/\phi(p^e)=pk/(p-1)$. From Lagrange's Theorem we obtain $|N|\le p^{1+\lfloor\log_p(k/(p-1))\rfloor}$. Now the first claim follows since $|G|=|G:\N_G(X)||\N_G(X):N||N|$.

Next let $e=2$ and $p\le 3$. By theorems of Wall~\cite{Wall} and Laffey~\cite{Laffeyp3} the number of elements of order $p$ is less than $3|G|/4$ for $p=2$ and less than $7|G|/9$ for $p=3$. This implies $|G|-k\le 3|G|/4$ and $|G|-k\le 7|G|/9$ respectively. The second claim follows.
\end{proof}

For odd $p$, the first bound in \autoref{podd} is best possible for cyclic groups and groups of exponent $p$. For $p=2$ however, the bound is optimal for the $2$-groups of maximal nilpotency class, i.\,e. for the dihedral, semidihedral and quaternion groups. We do not know if there are other groups attaining the bound.

The following lemma is taken from \cite[Lemma~8]{maxel30}. For the convenience of the reader we provide a proof.

\begin{Lem}\label{lemkm}
With the notation above there exists an integer $a$ such that $|G|$ divides $km^a$. 
\end{Lem}
\begin{proof}
Let $p$ be a prime divisor of $|G|$ which does not divide $m$. Then $p$ does not divide $|\C_G(x)|$ either, since otherwise there would be an element of order $pm$. Hence, $p$ divides the size $|G:\C_G(x)|$ of the conjugacy class of $x$. The union of all those conjugacy classes is the set of elements of order $m$. Therefore, $p$ divides $k$.
\end{proof}

The bound in \autoref{main} is often large when $m$ is small. The next observation excludes many exceptional cases. 

\begin{Lem}\label{lem2p}
With the notation above let $m=2p$ where $p$ is an odd prime. Then $|G|\le 2k(k+1)$.
\end{Lem}
\begin{proof} 
For $y\in\C_G(X)\setminus X$ we have $\langle x,y\rangle\le C_{2p}\times C_{2p}$. So there are at least $p$ elements of order $2p$ in the coset $yX$. It follows that $\frac{1}{2}|\C_G(X)|-1\le k$ and $|G|=|G:\C_G(X)||\C_G(X)|\le 2k(k+1)$.
\end{proof}

\begin{Prop}\label{lem28}
If $|G|\in\{2^9,2^{10},2^9\cdot 3\}$, then $k\ge 36$.
\end{Prop}
\begin{proof}
Suppose that $k<36$. First we consider $|G|=2^9$. Any element $x\in G$ of maximal order $m$ must lie in some maximal subgroup $M<G$ (otherwise $G$ is cyclic and $k=2^8>35$). By making use of the small groups library in GAP~\cite{GAP48}, it turns out that $M$ has exactly $32$ elements of order $m$ (there are no groups with fewer elements of maximal order). Moreover, there are just $89$ candidates for $M$ up to isomorphism. With the \texttt{grpconst} package in GAP we show that none of those extends to a group of order $2^9$ with $k<36$. By an inductive argument there are no groups of order $2^{10}$ with $k<36$. 

Now let $|G|=2^9\cdot 3$. Here $m$ is divisible by $3$, since otherwise we get examples of order $2^9$. Hence, $X$ lies in the centralizer of a Sylow $3$-subgroup $Q$ of $G$. In almost all cases, $Q$ is normal in $G$ and $|G:\C_G(Q)|\le 2$. However, a Sylow $2$-subgroup of $\C_G(Q)$ has at least $32$ elements of maximal order (see above). Therefore, $\C_G(Q)$ has at least $64$ elements of maximal order. Consequently, $Q$ is not normal in $G$. By consulting the \texttt{SmallGroupsInformation} command in GAP, we learn that there are „only“ $114{,}464$ such groups which can be checked in a matter of hours. 
\end{proof}

\begin{Thm}\label{k31}
There are exactly $13007$ finite groups with $k<36$ elements of maximal order $m$. Just $10684$ of those occur for $k=32$. 
The distribution of the group orders for $k\ne32$ is given in the appendix. 
\end{Thm}
\begin{proof}
If $k$ is odd, so is $\phi(m)$ and it follows that $m\in\{1,2\}$. It is well-known that $G$ must be an elementary abelian $2$-group in this situation. In particular, $|G|=1=k$ or $k=|G|-1=2^n-1$ for $1\le n\le 5$. Therefore, we may restrict ourselves to even values of $k$. It is straight forward to determine (by computer) for each $k$ the possible integers $m$ such that $\phi(m)$ divides $k$. 
If $k<24$, then \autoref{main} yields $|G|\le 2000$. These groups can be enumerated quickly with GAP~\cite{GAP48} by taking \autoref{lem28} into account. It is perhaps surprising that there are no groups with $k=34$.

For the remainder of the proof we focus on parameters where $|G|>2000$. If $k$ does not divide $|G|/m$, then there are at least two conjugacy classes of elements of order $m$. Hence, we may choose $X$ such that $|G:\C_G(X)|\le k/2$. Together with the coset counting argument $|\C_G(X)|\le mk/\phi(m)$ from the proof of \autoref{main}, many cases can be excluded in an automatic fashion. Moreover, we use the results obtained above with comment.
The remaining cases are handled by ad hoc arguments:

\begin{enumerate}[(i)]
\item $k=24$, $m=84=2^2\cdot 3\cdot 7$ and $|G|=mk^2/\phi(m)=2016$. Here the structure of $G$ is fairly restricted by \autoref{sharp}. In particular, the unique $X$ is contained in a subgroup $H\le G$ of index $2$. Since we have already determined all candidates for $H$, the possible extensions $G$ can be obtained with the \texttt{grpconst} package in GAP. It turns out that there are exactly $32$ such groups of order $2016$.

\item $k=28$, $m=12$ and $|G|=2016=2^5\cdot 3^2\cdot 7$. It is clear that $|\C_G(X)|$ is not divisible by $7$.
Suppose first that there exists some $X$ such that $|G:\C_G(X)|=k$. It can be checked by GAP that $\C_G(X)$ is isomorphic to $C_{12}\times C_6$ or to $C_{12}\times S_3$. However, the former group already contains $32$ elements of order $12$. Thus, $\C_G(X)\cong C_{12}\times S_3$ and all $28$ elements of order $12$ lie in $\C_G(X)$. Since $\C_G(X)$ is generated by those elements, it follows that $\C_G(X)\unlhd G$. However, $X=\Z(\C_G(X))$ is characteristic in $\C_G(X)$, so the generators of $X$ are not conjugate to elements outside $X$. This contradiction implies that we find $X$ with $|G:\C_G(X)|\le 14$. Coset counting gives $|\C_G(X)|\le 3k$ and $|G|\le 42k =1176$, a contradiction.

\item $k=32$, $m=12$ and $|G|=2304=2^8\cdot 3^2$. Since $|\C_G(X)|\le mk/\phi(m)=96$, we must have $|G:\C_G(X)|=k$ and $|\C_G(X)|=72$. As in the previous case, $\C_G(X)\cong C_{12}\times C_6$ or $\C_G(X)\cong C_{12}\times S_3$. Suppose first that $\C_G(X)\cong C_{12}\times C_6$. Then $\C_G(X)\unlhd G$ and there exists a subgroup $\C_G(X)\le H\le G$ such that $|G:H|=2$. There are $81$ candidates for $H$ (with $(k,m)=(32,12)$), but none of those can be extended to a group of order $2^8\cdot 3^2$ with $k=32$. Next let $\C_G(X)=C_{12}\times S_3$. Here $\C_G(X)$ contains only $24$ elements of order $m$. Hence, $X$ is conjugate to some subgroup $Y$ outside $\C_G(X)$. Moreover, $\C_G(X)$ and $\C_G(Y)$ have at least $16$ elements of order $m$ in common. 
This yields $\C_G(X)\cap\C_G(Y)\cong C_{12}\times C_3$. It follows easily that $G$ has a normal Sylow $3$-subgroup. Again there must be some $\C_G(X)\le H\le G$ with $|G:H|=2$. This was already dismissed above.

\item $k=32$, $m=30$ and $|G|=2880=2^6\cdot 3^2\cdot 5$. Since $|\C_G(X)|\le mk/\phi(m)=120$, we must have $|G:\C_G(X)|=k$ and $\C_G(X)\cong C_{30}\times C_3$. Hence, all elements of order $m$ lie in $\C_G(X)$ and $\C_G(X)\unlhd G$. Consequently, there exists $\C_G(X)\le H\le G$ with $|G:H|=2$. However, there are no candidates for $H$.

\item $k=32$, $m=68=2^2\cdot 17$ and $|G|=2176=2^7\cdot 17=mk^2/\phi(m)$. Here the structure of $G$ is determined by \autoref{sharp}. As in the first case, the desired groups can be constructed explicitly. There are $8$ of them.

\item $k=32$, $m=102=2\cdot 3\cdot 17$ and $|G|=3264=2^6\cdot 3\cdot 17=mk^2/\phi(m)$. Again \autoref{sharp} applies and there are just $4$ such groups.\qedhere
\end{enumerate}
\end{proof}

The following corollary implies a special case of Thompson's Conjecture mentioned in the introduction (cf. \cite[Proposition~1]{maxel30}).

\begin{Cor}
Let $G$ and $H$ be finite groups of the same order and the same maximal element order $m$. Suppose that $G$ and $H$ both have exactly $k<36$ elements of order $m$. Then $G$ is solvable if and only if $H$ is solvable. 
\end{Cor}
\begin{proof}
Since we have computed the groups in \autoref{k31} explicitly, it is easy to extract the non-solvable ones:
\begin{itemize}
\item $(k,m)=(20,6)$: $G=S_5$,
\item $(k,m)=(24,5)$: $G=A_5$,
\item $(k,m)=(24,10)$: $G=A_5\times C_2$, $\SL(2,5)$, $S_5\times C_2$ and $\SL(2,5)\rtimes C_2$.
\end{itemize}
The claim now follows by inspection of the table in the appendix.
\end{proof}

\section*{Acknowledgment}
The first author is supported by the German Research Foundation (\mbox{SA 2864/1-2} and \mbox{SA 2864/3-1}).

\section*{Appendix}

The following table contains the parameters of groups $G$ with exactly $k$ elements of maximal order $m$ where $k<36$ and $k\ne 32$. An entry of the form $n^s$ means that there are $s$ groups of order $n$ up to isomorphism. The small group ids (if available) of all groups including $k=32$ can be accessed on the first author's homepage.

\begin{center}
\begin{footnotesize}
\begin{longtable}[]{cclc}
\toprule
$k$&$m$&$|G|^{\#}$&total\\\toprule\endhead
1&1&$1$&1\\
&2&$2$&1\\\midrule
2&3&$3,6$&2\\
&4&$4,8$&2\\
&6&$6,12^{2}$&3\\\midrule
3&2&$4$&1\\\midrule
4&4&$8,16$&2\\
&5&$5,10,20$&3\\
&8&$8,16^{3}$&4\\
&10&$10,20^{2},40^{2}$&5\\
&12&$12,24^{5},48^{4}$&10\\\midrule
6&4&$8,24$&2\\
&6&$12,18,24^{3},36,72$&7\\
&7&$7,14,21,42$&4\\
&9&$9,18$&2\\
&14&$14,28^{2},42,84^{2}$&6\\
&18&$18,36^{2}$&3\\\midrule
7&2&$8$&1\\\midrule
8&3&$9,12,18$&3\\
&4&$16^{3},32$&4\\
&6&$18,24^{2},36^{2},48^{2},72,144^{2}$&10\\
&8&$16^{2},32^{11}$&13\\
&12&$24,48^{11},96^{16}$&28\\
&15&$15,30^{3},60^{3},120$&8\\
&16&$16,32^{3}$&4\\
&20&$20,40^{5},80^{9},160^{4}$&19\\
&24&$24,48^{9},96^{14}$&24\\
&30&$30,60^{6},120^{12},240^{8}$&27\\\midrule
10&11&$11,22,55,110$&4\\
&22&$22,44^{2},110,220^{2}$&6\\\midrule
12&4&$16^{3},32^{3},48$&7\\
&6&$36$&1\\
&8&$48^{2}$&2\\
&10&$20,40^{3},80^{4}$&8\\
&12&$24,36,48^{4},72^{5},144^{5}$&16\\
&13&$13,26,39,52,78,156$&6\\
&21&$21,42^{3},63^{2},84,126^{4},252$&12\\
&26&$26,52^{2},78,104^{2},156^{2},312^{2}$&10\\
&28&$28,56^{5},84,112^{4},168^{5},336^{4}$&20\\
&36&$36,72^{5},144^{4}$&10\\
&42&$42,84^{6},126^{2},168^{8},252^{8},504^{8}$&33\\\midrule
14&6&$24,48^{4}$&5\\\midrule
15&2&$16$&1\\\midrule
16&4&$32^{4},64$&5\\
&8&$32^{9},64^{85},128^{9}$&103\\
&12&$36,48^{5},72^{6},96^{47},144^{10},192^{36},288^{14},576^{11}$&130\\
&16&$32^{2},64^{16}$&18\\
&17&$17,34,68,136,272$&5\\
&20&$40,80^{11},160^{32},320^{25}$&69\\
&24&$48^{2},96^{37},192^{210}$&249\\
&32&$32,64^{3}$&4\\
&34&$34,68^{2},136^{2},272^{2},544^{2}$&9\\
&40&$40,80^{9},160^{23},320^{20}$&53\\
&48&$48,96^{9},192^{14}$&24\\
&60&$60,120^{13},240^{54},480^{76},960^{32}$&176\\\midrule
18&4&$36$&1\\
&6&$54^{3},108$&4\\
&9&$27^{2},54^{3},81,162^{2}$&8\\
&14&$28,56^{3},84^{2},168^{4}$&10\\
&18&$36,54^{3},72^{3},108^{9},162,216^{3},324^{4}$&24\\
&19&$19,38,57,114,171,342$&6\\
&27&$27,54$&2\\
&38&$38,76^{2},114,228^{2},342,684^{2}$&9\\
&54&$54,108^{2}$&3\\\midrule
20&4&$32^{5}$&5\\
&6&$36,120$&2\\
&10&$50,100$&2\\
&25&$25,50,100$&3\\
&33&$33,66^{3},132,165,330^{3},660$&10\\
&44&$44,88^{5},176^{4},220,440^{5},880^{4}$&20\\
&50&$50,100^{2},200^{2}$&5\\
&66&$66,132^{6},264^{8},330,660^{6},1320^{8}$&30\\\midrule
22&23&$23,46,253,506$&4\\
&46&$46,92^{2},506,1012^{2}$&6\\\midrule
24&4&$32^{7},64^{3}$&10\\
&5&$25,50,60,75,100^{2},200$&7\\
&6&$36,48^{2},54,72^{4},96^{3},108^{2},144^{2},216,288^{4}$&20\\
&8&$32,96^{5}$&6\\
&9&$36,72$&2\\
&10&$50,100^{2},120^{2},150,200^{4},240^{2},300,400^{3},600,1200$&18\\
&12&$48^{3},72^{3},96^{44},144^{12},192^{38},288^{19},576^{14}$&133\\
&18&$72^{2},144^{4}$&6\\
&20&$40,80^{4},120,160^{4},240,480$&12\\
&24&$72,144^{12},288^{24}$&37\\
&28&$56,112^{11},168,224^{16},336^{11},672^{16}$&56\\
&30&$60,90,120^{9},180^{3},240^{29},360^{7},480^{27},720^{6},1440^{2}$&85\\
&35&$35,70^{3},105,140^{3},210^{3},280,420^{3},840$&16\\
&36&$72,144^{11},288^{16}$&28\\
&39&$39,78^{3},117^{2},156^{3},234^{4},312,468^{4},936$&19\\
&45&$45,90^{3},180^{3},360$&8\\
&52&$52,104^{5},156,208^{9},312^{5},416^{4},624^{9},1248^{4}$&38\\
&56&$56,112^{9},168,224^{14},336^{9},672^{14}$&48\\
&70&$70,140^{6},210,280^{12},420^{6},560^{8},840^{12},1680^{8}$&54\\
&72&$72,144^{9},288^{14}$&24\\
&78&$78,156^{6},234^{2},312^{12},468^{8},624^{8},936^{14},1872^{8}$&59\\
&84&$84,168^{13},252^{2},336^{44},504^{18},672^{32},1008^{48},2016^{32}$&190\\
&90&$90,180^{6},360^{12},720^{8}$&27\\\midrule
26&3&$27^{2},54$&3\\
&6&$54^{2},108^{2}$&4\\\midrule
28&4&$32^{3},64$&4\\
&10&$40,80^{4},160^{7}$&12\\
&12&$72,144^{3}$&4\\
&29&$29,58,116,203,406,812$&6\\
&58&$58,116^{2},232^{2},406,812^{2},1624^{2}$&10\\\midrule
30&6&$48,72,96^{5},144,288$&9\\
&22&$44,88^{3},220,440^{3}$&8\\
&31&$31,62,93,155,186,310,465,930$&8\\
&62&$62,124^{2},186,310,372^{2},620^{2},930,1860^{2}$&12\\\midrule
31&2&$32$&1\\\bottomrule
\end{longtable}
\end{footnotesize}
\end{center}

\end{document}